\documentclass[12pt]{article}
\usepackage{amsmath,amssymb,amscd,latexsym,amsthm}

\newtheorem{Theorem}{Theorem}[section]

\newtheorem{Proposition}[Theorem]{Proposition}
\newtheorem{Lemma}[Theorem]{Lemma}

\newtheorem{Remark}[Theorem]{Remark}

\begin{document}

\title{Backward uniqueness for general parabolic operators in the whole space}

\author{Jie Wu\thanks{\small Center for Applied Mathematics, Tianjin University, Tianjin 300072, PR China;
E-mail: jackwu@amss.ac.cn} \,and\, Liqun Zhang\thanks{\small Hua Loo-Keng Key Laboratory of Mathematics, Institute of Mathematics, AMSS; School of Mathematical Sciences, UCAS, Beijing 100190, PR China; E-mail: lqzhang@math.ac.cn}}

\date{October, 2017}

\maketitle

\begin{abstract}
We prove the backward uniqueness for general parabolic operators of second order in the whole space
under assumptions that the leading coefficients of the operator are Lipschitz
and their gradients satisfy certain decay conditions. This result
extends in some ways a classical result of Lions and Malgrange \cite{LM} and a recent result of the authors \cite{WZ1}.
\end{abstract}



\section{Introduction}

Let $\emph{P}$ be a backward parabolic operator on $\mathbb{R}^n\times[0,1]$,
$$
P=\partial_t+\partial_i\big(a^{ij}(x,t)\partial_j\big)=\partial_t+\nabla\cdot(\mathbf{A}\nabla),
$$
where $\mathbf{A}(x,t)=(a^{ij}(x,t))^n_{i,j=1}$ is a real symmetric matrix
such that for some $\Lambda\geq\lambda>0$,
\begin{equation}\label{a1}
\lambda|\xi|^2\leq a^{ij}(x,t)\xi_i\xi_j\leq
\Lambda|\xi|^2, ~~~\forall~ \xi\in\mathbb{R}^n.
\end{equation}
Here we work with backward parabolic operators because it is more convenient in this context.
A function $u$ satisfies that
$$
|Pu|\leq N(|u|+|\nabla u|)
$$
and the growth condition
\begin{equation*}\label{growth}
|u(x,t)|\leq Ne^{N|x|^\alpha}
\end{equation*}
for some $\alpha\in[0,2]$, or a weaker condition
\begin{equation*}\label{integ-con}
e^{-N|x|^\alpha}u(x,t)\in L^2(\mathbb{R}^n\times[0,1]).
\end{equation*}

The backward uniqueness problem is: suppose
$$
u(x,0)=0,~~~\forall~x\in\mathbb{R}^n,
$$
does $u$ vanish identically in $\mathbb{R}^n\times[0,1]$?

Here we set $\alpha\in[0,2]$ because the classical examples of Tychonoff \cite{Ty}
show that the backward uniqueness fails when $\alpha>2$.

The backward uniqueness problem has a natural background in the control
theory for PDEs. It also appeared in the regularity theory of parabolic equations,
for example, it was applied to prove the full regularity of $L_{3,\infty}$-solutions of the 3-dimensional Navier-Stokes equations \cite{ESS}.

When $P$ is with constant coefficients, i.e., the backward heat operator, there are already many
results in various domains, such as the exterior domain \cite{ESS1}, the half space \cite{ESS2} and some cones \cite{LS,WW,R¨¹}.

When $P$ is a general operator with variable coefficients,
the results for the Landis-Oleinik conjecture \cite{Ngu,WZ2},
backward uniqueness in the half space \cite{WZ1} and unique continuation \cite{EF}
imply that the backward uniqueness in the whole space is valid under the Lipschitz conditions
\begin{equation}\label{Nguyen1}
|\nabla_xa^{ij}(x,t)|+|\partial_ta^{ij}(x,t)|\leq M,
\end{equation}
and the decay at infinity conditions
\begin{equation}\label{opcond}
|\nabla_xa^{ij}(x,t)|\leq E|x|^{-1},~~~where~~E<E_0(n,\Lambda,\lambda).
\end{equation}
Moreover, both conditions (\ref{Nguyen1}) and (\ref{opcond}) are almost optimal
for the backward uniqueness in the half space when the growth rate of $u$ is quadratic exponential (i.e. $\alpha=2$),
which could be seen from the counter examples constructed by the authors in \cite{WZ1}.

On the other hand, the classical result of Lions and Malgrange \cite{LM} showed that the backward uniqueness is valid
if $u$ lies in the space
$$\mathcal{H}:=H^1\big((0,1),L^2(\mathbb{R}^n_x)\big)\cap L^2\big((0,1),H^2(\mathbb{R}^n_x)\big)$$
and $$a^{ij}(x,t)\in Lip\big([0,1],L^\infty(\mathbb{R}^n_x)\big).$$

In this paper we will prove a result which extends the above two results in some ways. We observe that there is a link between
the decay rate of $|\nabla_xa^{ij}(x,t)|$ and the exponential growth rate of $u$. We denote
$$\langle x\rangle=\sqrt{1+|x|^2},~~~ \beta=\max{\{0,\alpha-1\}}.$$ Our main result is the following.
\begin{Theorem}\label{mainthm}
Suppose $\{a^{ij}\}$ satisfy (\ref{a1}), and for some constants $E,M,N>0$,
\begin{equation}\label{a2}
|\nabla_xa^{ij}(x,t)|+|\partial_ta^{ij}(x,t)|\leq M,~~~~~|\nabla_xa^{ij}(x,t)|\leq E\langle x\rangle^{-\beta}
\end{equation}
in $\mathbb{R}^n\times[0,1]$. Assume that $u$ satisfies
\begin{equation}\label{u-cond1}
|Pu|\leq N(|u|+|\nabla u|)
\end{equation}
and
\begin{equation}\label{u-cond2}
|u(x,t)|\leq Ne^{N|x|^\alpha}~~~or~~~e^{-N|x|^\alpha}u(x,t)\in L^2(\mathbb{R}^n\times[0,1]).
\end{equation}
Then if $u(x,0)=0$ in $\mathbb{R}^n$, $u$ vanishes identically in $\mathbb{R}^n\times[0,1]$.
\end{Theorem}

\begin{Remark}\mbox{}
\begin{enumerate}

  \item When $\alpha\in[0,1]$, Theorem \ref{mainthm} tells us that the Lipschitz conditions (\ref{Nguyen1}) are sufficient for the backward uniqueness even if $|u(x,t)|\leq Ne^{N|x|}$ or $e^{-N|x|}u(x,t)\in L^2(\mathbb{R}^n\times[0,1]).$
This extends the result of Lions and Malgrange \cite{LM} in some ways.
  \item When $\alpha=2$, it required the smallness of $E$ for the backward uniqueness in the half space \cite{WZ1}.
However as for the whole space, we don't require such condition.
\end{enumerate}
\end{Remark}

To prove our result we need the following Carleman inequality.
\begin{Proposition}\label{Prop-C1}
Suppose $\{a^{ij}\}$ satisfy (\ref{a1}) and (\ref{a2}).
For any $v\in C^\infty_0(\mathbb{R}^n\times(0,1))$ and any $\gamma>0$, we have
\begin{equation}\label{CI1}
\begin{split}
\int_{\mathbb{R}^n\times (0,1)}e^{2\gamma(t^{-K}-1)-\frac{b\langle x\rangle^\alpha+K}{t}}(|v|^2+|\nabla v|^2)dxdt\\
\leq \int_{\mathbb{R}^n\times (0,1)}e^{2\gamma(t^{-K}-1)-\frac{b\langle x\rangle^\alpha+K}{t}}|Pv|^2dxdt,
\end{split}
\end{equation}
where $b=\frac{1}{8\Lambda}$ and $K=K(n,\Lambda,\lambda,M,E,\alpha)$.
\end{Proposition}

It is worthwhile to mention \cite{SP,SJP} and related results, which discuss the backward uniqueness problem when $u\in\mathcal{H}$ and
$a^{ij}(x,t)$ are non-Lipschitz. However, here we just assume that u satisfies (\ref{u-cond2}).

The paper is organized as follows. First we use Carleman inequality (\ref{CI1})
to prove Theorem \ref{mainthm}, then we prove this Carleman inequality.

\section{Proof of the main result}

In this section, we prove Theorem \ref{mainthm}. First, we extend $u$ and $a^{ij}$ by the following way:
\begin{equation*}
\begin{aligned}
u(x,t)=&0,~~~~~&&if~~t<0;\\
a^{ij}(x,t)=&a^{ij}(x,0),~~~&&if~~t<0.
\end{aligned}
\end{equation*}
The next lemma implies Theorem \ref{mainthm} immediately.
\begin{Lemma}
Suppose $\{a^{ij}\}$ and $u$ are the same as those in Theorem \ref{mainthm}. Then there exists
$T_1=T_1(\Lambda,N)>0$, such that
$u(x,t)\equiv0$
in $\mathbb{R}^n\times(0,T_1)$.
\end{Lemma}

\begin{proof}
We use Carleman inequality (\ref{CI1}) to prove
this lemma, mainly following the arguments of the corresponding parts in \cite{ESS1} and \cite{WZ1}.
We just give the proof for the case that $|u(x,t)|\leq Ne^{N|x|^\alpha}$, since the proof of the other case is similar.

Without loss of generality, we assume that $\alpha \in [1,2]$.

\textbf{Step 1.} By the regularity theory for solutions
of parabolic equations, we have
\begin{equation}\label{pt1}
|u(x,t)|+|\nabla u(x,t)|\leq C(n,\Lambda,\lambda,M,N)e^{2N|x|^{\alpha}}
\end{equation}
when $(x,t)\in \mathbb{R}^n\times(0,\frac{1}{2})$.
In the following, we always denote $C=C(n,\Lambda,\lambda,M,N)$.
Let
\begin{equation}\label{choice-tau}
\tau=\min\{1,\frac{1}{2N},\frac{b}{8N}\}.
\end{equation}
We denote
$$\tilde u(x,t)=u\big(\tau x,\tau^2(t-\frac{1}{2})\big)$$
and
$$\tilde{a}^{ij}(x,t)={a}^{ij}\big(\tau x,\tau^2(t-\frac{1}{2})\big)$$
for $(x,t)\in\mathbb{R}^n\times(0,1)$. Then it is easy to see that
$$|\nabla_x\tilde{a}^{ij}(x,t)|+|\partial_t\tilde{a}^{ij}(x,t)|\leq \tau M\leq M,$$
and
$$|\nabla_x\tilde{a}^{ij}(x,t)|=\tau|\nabla_x a^{ij}\big(\tau x, \tau^2(t-\frac{1}{2})\big)|
\leq \tau E\langle \tau x\rangle^{-\beta}\leq E\tau^{1-\beta}\langle x\rangle^{-\beta}\leq E\langle x\rangle^{-\beta}.$$
We denote
$$\tilde{P}=\partial_t+\partial_i(\tilde{a}^{ij}\partial_j),$$
then by (\ref{u-cond1}) and (\ref{choice-tau}) we have
\begin{equation}\label{pt2}
|\tilde{P}\tilde u|\leq\tau N(|\tilde u|+|\nabla \tilde u|)\leq\frac{1}{2}(|\tilde u|+|\nabla \tilde u|).
\end{equation}
By (\ref{pt1}) and (\ref{choice-tau}) we have
\begin{equation}\label{pt3}
|\tilde u(x,t)|+|\nabla \tilde u(x,t)|\leq Ce^{2N\tau^\alpha|x|^\alpha}\leq Ce^{2N\tau |x|^\alpha}\leq Ce^{\frac{b}{4}\langle x\rangle^\alpha}
\end{equation}
when $(x,t)\in \mathbb{R}^n\times(0,1)$, and
\begin{equation}\label{pt4}
\tilde u(x,t)=0
\end{equation}
when $(x,t)\in \mathbb{R}^n\times(0,\frac{1}{2}]$.

\textbf{Step 2.} In order to apply Carleman inequality (\ref{CI1}), we choose two
smooth cut-off functions such that
$$
\eta_1(|x|)=\left\{
                 \begin{array}{ll}
                   1, & \hbox{if $|x|<R$;} \\
                   0, & \hbox{if $|x|>R+1$,}
                 \end{array}
               \right.
$$
where $R$ is large enough, and
$$
\eta_2(t)=\left\{
            \begin{array}{ll}
              1, & \hbox{if $t<\frac{3}{4}$;} \\
              0, & \hbox{if $t>\frac{7}{8}$.}
            \end{array}
          \right.
$$
Let $\eta=\eta_1\eta_2$ and $v=\eta \tilde u$. Then $supp~v\subset \mathbb{R}^n\times(0,1)$. By (\ref{pt2}) we have
\begin{equation}\label{pt5}
\begin{split}
|\tilde{P}v|=&|\eta \tilde{P}\tilde u+\tilde u\tilde{P}\eta+2\tilde{a}^{ij}\partial_i\eta\partial_j \tilde u|\\
\leq&\frac{1}{2}\eta(|\tilde u|+|\nabla\tilde u|)+C(|\tilde u|+|\nabla \tilde u|)(|\partial_t\eta|+|\nabla\eta|+|\nabla^2\eta|)\\
\leq&\frac{1}{2}(|v|+|\nabla v|)+C\chi_\Omega(|\tilde u|+|\nabla \tilde u|),
\end{split}
\end{equation}
where $\chi$ is the characteristic function and
$$\Omega=\{0<\eta<1,~\frac{1}{2}<t<1\}.$$
Moreover,
\begin{eqnarray*}
\begin{split}
\Omega=&\{0<\eta_1<1,~\eta_2>0,~\frac{1}{2}<t<1\}\cup\{~\eta_1=1,~0<\eta_2<1,~\frac{1}{2}<t<1\}\\
=&\{R<|x|<R+1,~\frac{1}{2}<t<\frac{7}{8})\}\cup\{|x|<R,~\frac{3}{4}<t<\frac{7}{8}\}.
\end{split}
\end{eqnarray*}

\textbf{Step 3.} We apply Carleman inequality (\ref{CI1}) for $\tilde{P}$ and $v$, then
\begin{eqnarray*}
\begin{split}
J&\equiv\int_{\mathbb{R}^n\times (0,1)}e^{2\gamma(t^{-K}-1)-\frac{b\langle x\rangle^\alpha+K}{t}}(|v|^2+|\nabla v|^2)dxdt\\
&\leq \int_{\mathbb{R}^n\times (0,1)}e^{2\gamma(t^{-K}-1)-\frac{b\langle x\rangle^\alpha+K}{t}}|\tilde Pv|^2dxdt.\\
\end{split}
\end{eqnarray*}
By (\ref{pt5}) we have
$$J\leq \frac{3}{4}J+C\int_\Omega e^{2\gamma(t^{-K}-1)-\frac{b\langle x\rangle^\alpha+K}{t}}(|\tilde u|+|\nabla \tilde u|)^2dxdt,$$
thus
$$
J\leq C\int_\Omega e^{2\gamma(t^{-K}-1)-\frac{b\langle x\rangle^\alpha+K}{t}}(|\tilde u|+|\nabla \tilde u|)^2dxdt.
$$
By (\ref{pt3}) we obtain
\begin{equation}\label{est-ini}
\begin{split}
J\leq&C\int_\Omega e^{2\gamma(t^{-K}-1)-\frac{b}{2}\langle x\rangle^\alpha}dxdt\\
=&C\Big(\int_{\{R<|x|<R+1,~\frac{1}{2}<t<\frac{7}{8}\}}
+\int_{\{|x|<R,~\frac{3}{4}<t<\frac{7}{8}\}}\Big)e^{2\gamma(t^{-K}-1)-\frac{b}{2}\langle x\rangle^\alpha}dxdt\\
\equiv&J_1+J_2.
\end{split}
\end{equation}

\textbf{Step 4.} Now we estimate both sides of the above inequality.\\

\textbf{Estimate of $J_1$.}
\begin{equation}\label{est-J1}
\begin{split}
J_1\leq&Ce^{2\gamma(2^K-1)}\int_{\{R<|x|<R+1\}}e^{-\frac{b}{2}\langle x\rangle^\alpha}dx\\
\leq&Ce^{2^{K+1}\gamma-\frac{b}{4}R^\alpha}\int_{\{R<|x|<R+1\}}e^{-\frac{b}{4}\langle x\rangle^\alpha}dx\\
\leq&Ce^{2^{K+1}\gamma-\frac{b}{4}R^\alpha}.
\end{split}
\end{equation}

\textbf{Estimate of $J_2$.}
\begin{equation}\label{est-J2}
J_2\leq Ce^{2\gamma[(\frac{3}{4})^{-K}-1]}\int_{\{|x|<R\}}e^{-\frac{b}{2}\langle x\rangle^\alpha}dx
\leq Ce^{2\gamma[(\frac{3}{4})^{-K}-1]}.
\end{equation}

\textbf{Estimate of $J$.}
For an arbitrary $l\in(\frac{1}{2},\frac{3}{4})$,
\begin{equation}\label{est-J}
\begin{split}
J\geq&\int_{\{|x|<R,~\frac{1}{2}<t<l\}}e^{2\gamma(t^{-K}-1)-\frac{b\langle x\rangle^\alpha+K}{t}}(|\tilde u|^2+|\nabla \tilde u|^2)dxdt\\
\geq&e^{2\gamma(l^{-K}-1)}\int_{\{|x|<R,~\frac{1}{2}<t<l\}}e^{-\frac{b\langle x\rangle^\alpha+K}{t}}(|\tilde u|^2+|\nabla \tilde u|^2)dxdt.
\end{split}
\end{equation}
We combine (\ref{est-ini})-(\ref{est-J}), then we have
\begin{equation*}
\begin{split}
&\int_{\{|x|<R,~\frac{1}{2}<t<l\}}e^{-\frac{b\langle x\rangle^\alpha+K}{t}}(|\tilde u|^2+|\nabla \tilde u|^2)dxdt\\
\leq&Ce^{2\gamma(1-l^{-K})}\big(e^{2^{K+1}\gamma-\frac{b}{4}R^\alpha}+e^{2\gamma[(\frac{3}{4})^{-K}-1]}\big).
\end{split}
\end{equation*}
In the above inequality, we fix $\gamma$ and let $R\rightarrow\infty$, then we obtain
$$
\int_{\mathbb{R}^n\times(\frac{1}{2},l)}e^{-\frac{b\langle x\rangle^\alpha+K}{t}}(|\tilde u|^2+|\nabla \tilde u|^2)dxdt
\leq Ce^{2\gamma[(\frac{3}{4})^{-K}-l^{-K}]}.
$$
Now we fix $l$ and let $\gamma\rightarrow\infty$, then we have
$\tilde u(x,t)\equiv0$ in $\mathbb{R}^n\times(\frac{1}{2},l)$.\\
Since $l$ is an arbitrary number in $(\frac{1}{2},\frac{3}{4})$, then $\tilde u(x,t)\equiv0$ in $\mathbb{R}^n\times(\frac{1}{2},\frac{3}{4})$.
That is, $u(x,t)\equiv 0$ in $\mathbb{R}^n\times(0,\frac{\tau^2}{4})$.\\
Finally we let
$$T_1=\frac{\tau^2}{4}=\min\{\frac{1}{4},\frac{1}{16N^2},\frac{b^2}{256N^2}\},$$
then $T_1=T_1(\Lambda, N)$ and $u(x,t)\equiv 0$ in $\mathbb{R}^n\times(0,T_1)$.

Thus we proved this lemma.
\end{proof}

\section{Proof of  the Carleman inequality}

In this section, we prove Carleman inequality (\ref{CI1}). We need two lemmas  in our proof. The first one is due to Escauriaza and Fern\'{a}ndez \cite{EF} (see also
\cite[Corollary 3.2]{WZ2}).
In the following, we denote
$$\tilde{\Delta}=\partial_i(a^{ij}\partial_j).$$

\begin{Lemma}\label{cor-generalCI1}
Suppose $F$ is differentiable, $F_0$ and $G$ are twice differentiable and $G>0$. Then the
following identity holds for any $v\in C^\infty_0(\mathbb{R}^n\times(0,T))$:
\begin{equation}\label{generalCI1}
\begin{split}
&\frac{1}{2}\int_{\mathbb{R}^n\times(0,T)}v^2M_0Gdxdt+
\int_{\mathbb{R}^n\times(0,T)}\langle [2\mathbf{D}_G+(\frac{\partial_tG-
\tilde{\Delta}G}{G}-F)\mathbf{A}]\nabla v,\nabla v\rangle Gdxdt\\
&-\int_{\mathbb{R}^n\times(0,T)}v\langle \mathbf{A}\nabla v,\nabla (F-F_0)\rangle Gdxdt=
2\int_{\mathbb{R}^n\times(0,T)}Lv(Pv-Lv)Gdxdt,
\end{split}
\end{equation}
where
$$Lv=\partial_tv-\langle \mathbf{A}\nabla v,\nabla{\log G}\rangle+\frac{Fv}{2},$$
$$M_0=\partial_tF+F(\frac{\partial_tG-\tilde{\Delta}G}{G}-F)+\tilde{\Delta}F_0-
\langle \mathbf{A}\nabla (F-F_0),\nabla{\log G}\rangle,$$
and
$$\mathbf{D}^{ij}_G=a^{ik}\partial_{kl}(\log G)a^{lj}
+\frac{\partial_l(\log G)}{2}(a^{ki}\partial_ka^{lj}+a^{kj}\partial_ka^{li}-
a^{kl}\partial_ka^{ij})+\frac{1}{2}\partial_ta^{ij}.
$$
\end{Lemma}
The second one is concerned with the properties of mollified $\{a^{ij}\}$.
\begin{Lemma}\label{Lem-mollify}
Suppose $\{a^{ij}\}$ satisfy (\ref{a1}) and (\ref{a2}). Let
$$a^{ij}_\epsilon(x,t)=\int_{\mathbb{R}^n}a^{ij}(x-y,t)\phi_\epsilon(y)dy,$$
where $\phi$ is a mollifier and $\epsilon=\frac{1}{2}$.
Then $\{a^{ij}_\epsilon\}$ satisfy:
\begin{equation}\label{mollify}
\begin{aligned}
&1)~\lambda|\xi|^2\leq a^{ij}_\epsilon(x,t)\xi_i\xi_j\leq \Lambda|\xi|^2, &&~~\forall \xi\in\mathbb{R}^n;\\
&2)~|\nabla a^{ij}_\epsilon(x,t)|\leq M;&&|\nabla a^{ij}_\epsilon(x,t)|\leq 2E\langle x\rangle^{-\beta}~~~when~|x|\geq1;\\
&3)~|a^{ij}_\epsilon(x,t)-a^{ij}(x,t)|\leq 2\Lambda ;&&|a^{ij}_\epsilon(x,t)-a^{ij}(x,t)|\leq E\langle x\rangle^{-\beta}~~when~|x|\geq1;\\
&4)~|\partial_{kl}a^{ij}_\epsilon(x,t)|\leq c(n)M;&&|\partial_{kl}a^{ij}_\epsilon(x,t)|\leq c(n)E\langle x\rangle^{-\beta}~~~when~|x|\geq1.\\
\end{aligned}
\end{equation}
\end{Lemma}

\begin{proof}[Proof of Lemma \ref{Lem-mollify}.]\mbox{}\\
$1)$ Obvious.\\
$2)$
$$|\nabla a^{ij}_\epsilon(x,t)|\leq\int_{\mathbb{R}^n}|\nabla a^{ij}(x-y,t)|\phi_\epsilon(y)dy
\leq M\int_{\mathbb{R}^n}\phi_\epsilon(y)dy=M,$$
and when $|x|\geq1$,
$$
|\nabla a^{ij}_\epsilon(x,t)|\leq\int_{\mathbb{R}^n}|\nabla a^{ij}(x-y,t)|\phi_\epsilon(y)dy
\leq E\int_{\mathbb{R}^n}\langle x-y\rangle^{-\beta}\phi_\epsilon(y)dy.
$$
Since $|x|\geq1$ and $|y|\leq \frac{1}{2}$, then $\langle x-y\rangle\geq\frac{1}{2}\langle x\rangle$ and thus
$$
|\nabla a^{ij}_\epsilon(x,t)|
\leq E2^{\beta}\langle x\rangle^{-\beta}\int_{\mathbb{R}^n}\phi_\epsilon(y)dy
\leq 2E\langle x\rangle^{-\beta}.
$$
$3)$ The first part is obvious. We only need to prove the second
one.
\begin{equation*}
\begin{split}
|a^{ij}_\epsilon(x,t)-a^{ij}(x,t)|&\leq\int_{\mathbb{R}^n}|a^{ij}(x-y,t)-a^{ij}(x,t)|\phi_\epsilon(y)dy\\
&\leq\int_{\mathbb{R}^n}|\nabla a^{ij}(x-\theta y,t)||y|\phi_\epsilon(y)dy, ~~~~~~~(0<\theta<1)\\
\end{split}
\end{equation*}
and when $|x|\geq1$,
$$
|a^{ij}_\epsilon(x,t)-a^{ij}(x,t)|\leq\frac{E}{2}\int_{\mathbb{R}^n}\langle x-\theta y\rangle^{-\beta}\phi_\epsilon(y)dy
\leq E2^{\beta-1}\langle x\rangle^{-\beta}\int_{\mathbb{R}^n}\phi_\epsilon(y)dy
\leq E\langle x\rangle^{-\beta}.
$$
$4)$
\begin{equation*}
\begin{split}
|\partial_{kl}a^{ij}_\epsilon(x,t)|&\leq\int_{\mathbb{R}^n}|
\partial_ka^{ij}(x-y,t)||\partial_l\phi_\epsilon(y)|dy\\
&\leq\epsilon^{-n-1}\int_{\mathbb{R}^n}|\partial_ka^{ij}(x-y,t)||(\partial_l\phi)(\frac{y}{\epsilon})|dy\\
&\leq\frac{M}{\epsilon}\|\partial_l\phi\|_{L^1}\leq 2M\|\nabla\phi\|_{L^1},
\end{split}
\end{equation*}
and when $|x|\geq1$,
\begin{equation*}
\begin{split}
|\partial_{kl}a^{ij}_\epsilon(x,t)|&\leq\epsilon^{-n-1}\int_{\mathbb{R}^n}|
\partial_ka^{ij}(x-y,t)||(\partial_l\phi)(\frac{y}{\epsilon})|dy\\
&\leq\epsilon^{-n-1}E \int_{\mathbb{R}^n}\langle x-y\rangle^{-\beta}|(\partial_l\phi)(\frac{y}{\epsilon})|dy\\
&\leq \frac{E2^\beta}{\epsilon}\langle x\rangle^{-\beta}\|\partial_l\phi\|_{L^1}\leq 4E\langle x\rangle^{-\beta}\|\nabla\phi\|_{L^1}.
\end{split}
\end{equation*}
\end{proof}

Now we begin to prove Proposition \ref{Prop-C1}.

\begin{proof}[Proof of Proposition \ref{Prop-C1}.]
In (\ref{generalCI1}),
we let $$G=e^{2\gamma (t^{-K}-1)-\frac{b\langle x\rangle^\alpha+K}{t}},$$
then
$$\frac{\partial_tG-\tilde{\Delta}G}{G}=\frac{b\langle x\rangle^\alpha-\alpha^2b^2\langle x\rangle^{2\alpha-4}a^{ij}x_ix_j+K}{t^2}
+\frac{\alpha b\langle x\rangle^{\alpha-2}(a^{ii}+\partial_ka^{kl}x_l)}{t}-2\gamma Kt^{-K-1}.$$ Let
$$F=\frac{b\langle x\rangle^\alpha-\alpha^2b^2\langle x\rangle^{2\alpha-4}a^{ij}x_ix_j+K}{t^2}
+\frac{\alpha b\langle x\rangle^{\alpha-2}a^{ii}-d}{t}-2\gamma Kt^{-K-1},$$ where $d$ is a positive constant to be
determined, and
$$F_0=\frac{b\langle x\rangle^\alpha-\alpha^2b^2\langle x\rangle^{2\alpha-4}a^{ij}_\epsilon x_ix_j+K}{t^2}
+\frac{\alpha b\langle x\rangle^{\alpha-2}a^{ii}_\epsilon-d}{t}-2\gamma Kt^{-K-1}.$$

We denote by $\mathbf{I}_n$ the identity matrix of $\mathbb{R}^n$, $C$ are
generic constants depending on $n,\Lambda,\lambda,M,E$ and $\alpha$ in the
following arguments. We need some estimates which we list in the
following lemma.
\begin{Lemma}\label{estimates1}
Set $b=\frac{1}{8\Lambda}$ and $d=\frac{K}{4}$. For
$K\geq K_0(n,\Lambda,\lambda,M,E,\alpha)$, we have
\begin{align}
2\mathbf{D}_G+(\frac{\partial_t G-\tilde{\Delta}G}{G}-F)\mathbf{A}\geq &\frac{\lambda K}{8t}\mathbf{I}_n;\label{E1}\\
\partial_tF+F(\frac{\partial_tG-\tilde{\Delta}G}{G}-F)\geq &\frac{bK\langle x\rangle^{\alpha}}{16 t^3};\label{E2}\\
|\tilde{\triangle}F_0|\leq&\frac{C\langle x\rangle^{\alpha}}{t^2};\label{E3}\\
|\nabla(F-F_0)|\leq&\frac{C\langle x\rangle^{\alpha-1}}{t^2}.\label{E4}
\end{align}
\end{Lemma}
We will prove this lemma later.

First by (\ref{E1}) we have
\begin{equation}\label{pc1.1}
\begin{split}
&\int_{\mathbb{R}^n\times(0,1)}\langle [2\mathbf{D}_G+(\frac{\partial_tG-
\tilde{\Delta}G}{G}-F)\mathbf{A}]\nabla v,\nabla v\rangle Gdxdt\\
\geq&\frac{\lambda K}{8}\int_{\mathbb{R}^n\times(0,1)}\frac{|\nabla v|^2}{t}Gdxdt.
\end{split}
\end{equation}
Next we estimate $M_0$. By (\ref{E4}) and
$$\nabla \log G=-\frac{\alpha b}{t}\langle x\rangle^{\alpha-2}x$$ we have
\begin{equation}\label{pc1.5}
|\langle \mathbf{A}\nabla(F-F_0),\nabla \log G\rangle|\leq\Lambda|\nabla(F-F_0)||\nabla \log G|\leq
\frac{C\langle x\rangle^{2\alpha-2}}{t^3}\leq\frac{C\langle x\rangle^{\alpha}}{t^3} .
\end{equation}
Then by (\ref{E2}), (\ref{E3}) and (\ref{pc1.5}) we have
\begin{eqnarray*}
\begin{split}
M_0&=\partial_tF+F(\frac{\partial_tG-\tilde{\Delta}G}{G}-F)+\tilde{\Delta}F_0-
\langle A\nabla (F-F_0),\nabla{\log G}\rangle\\
&\geq (\frac{bK}{16}-C)\frac{\langle x\rangle^{\alpha}}{t^3},
\end{split}
\end{eqnarray*}
thus
\begin{equation}\label{pc1.6}
\frac{1}{2}\int_{\mathbb{R}^n\times(0,1)}v^2M_0Gdxdt
\geq (\frac{bK}{32}-C)\int_{\mathbb{R}^n\times(0,1)}\frac{\langle x\rangle^{\alpha}}{t^3}v^2Gdxdt.
\end{equation}
By the Cauchy inequality and (\ref{E4}) we have
\begin{equation}\label{pc1.7}
\begin{split}
&\Big|\int_{\mathbb{R}^n\times(0,1)}v\langle \mathbf{A}\nabla v,\nabla (F-F_0)\rangle Gdxdt\Big|\\
\leq& \Lambda\int_{\mathbb{R}^n\times(0,1)}|\nabla(F-F_0)||v||\nabla v|Gdxdt\\
\leq& C\int_{\mathbb{R}^n\times(0,1)}\frac{\langle x\rangle^{\alpha-1}}{t^2}|v||\nabla v|Gdxdt\\
\leq& C\int_{\mathbb{R}^n\times(0,1)}\frac{\langle x\rangle^{2\alpha-2}}{t^3}v^2Gdxdt
+C\int_{\mathbb{R}^n\times(0,1)}\frac{|\nabla v|^2}{t}Gdxdt\\
\leq& C\int_{\mathbb{R}^n\times(0,1)}\frac{\langle x\rangle^\alpha}{t^3}v^2Gdxdt
+C\int_{\mathbb{R}^n\times(0,1)}\frac{|\nabla v|^2}{t}Gdxdt.
\end{split}
\end{equation}
Finally, by (\ref{generalCI1}), (\ref{pc1.1}), (\ref{pc1.6}), (\ref{pc1.7}) and the Cauchy inequality, we have
$$
\int_{\mathbb{R}^n\times(0,1)}|Pu|^2Gdxdt
\geq (\frac{bK}{32}-C)\int_{\mathbb{R}^n\times(0,1)}\frac{\langle x\rangle^{\alpha}}{t^3} v^2Gdxdt+(\frac{\lambda K}{8}-C)\int_{\mathbb{R}^n\times(0,1)}\frac{|\nabla v|^2}{t}Gdxdt,\\
$$
if we choose $K\geq K_0(n,\Lambda,\lambda,M,E,\alpha)$ large enough, we obtain
$$
\int_{\mathbb{R}^n\times(0,1)}|Pv|^2Gdxdt
\geq \int_{\mathbb{R}^n\times(0,1)}(v^2+|\nabla v|^2)Gdxdt,
$$

Thus we proved Carleman inequality (\ref{CI1}).
\end{proof}

There is only Lemma \ref{estimates1} left to be proven.

\begin{proof}[Proof of Lemma \ref{estimates1}.] We estimate them one by one.\\

\textbf{Estimate of $2\mathbf{D}_G+(\frac{\partial_t G-\tilde{\Delta}G}{G}-F)\mathbf{A}$.}\\

By direct calculations we have
\begin{eqnarray*}
\begin{split}
&2\mathbf{D}_G+(\frac{\partial_t G-\tilde{\Delta}G}{G}-F)\mathbf{A}\\
=&-\frac{2\alpha b}{t}\langle x\rangle^{\alpha-2}\mathbf{A}^2+\frac{2\alpha(2-\alpha)b}{t}\langle x\rangle^{\alpha-4}\mathbf{A}x(\mathbf{A}x)^{'}\\
&-\frac{\alpha b}{t}\langle x\rangle^{\alpha-2}x_l(a^{ki}\partial_ka^{lj}+a^{kj}\partial_ka^{li}-a^{kl}
\partial_ka^{ij}-a^{ij}\partial_k a^{kl})+\partial_ta^{ij}+\frac{d}{t}\mathbf{A}\\
\geq&-\frac{2\alpha b\Lambda^2}{t}\langle x\rangle^{\alpha-2}\mathbf{I}_n-\frac{\alpha b}{t}\langle x\rangle^{\alpha-2}x_l(a^{ki}\partial_ka^{lj}+a^{kj}
\partial_ka^{li}-a^{kl}\partial_ka^{ij}-a^{ij}\partial_k a^{kl})+\partial_ta^{ij}+ \frac{\lambda d}{t}\mathbf{I}_n.
\end{split}
\end{eqnarray*}
Next we estimate the lower bounds of the matrices in the right side of the above inequality.
We just need to estimate matrix $x_la^{ki}\partial_ka^{lj}$ and
$\partial_ta^{ij}$. For any $\xi\in\mathbb{R}^n$,
$$|x_la^{ki}\partial_ka^{lj}\xi_i\xi_j|\leq n^2\Lambda E|x|\langle x\rangle^{-\beta}\sum_{i,j}|\xi_i||\xi_j|\leq
n^3\Lambda E\langle x\rangle^{1-\beta}|\xi|^2,$$
then
$$- n^3\Lambda E\langle x\rangle^{1-\beta}\mathbf{I}_n\leq x_la^{ki}\partial_ka^{lj}\leq n^3\Lambda E\langle x\rangle^{1-\beta}\mathbf{I}_n.$$
Similarly,
$$|\partial_ta^{ij}\xi_i\xi_j|\leq M\sum_{i,j}|\xi_i||\xi_j|\leq Mn|\xi|^2,$$
then
$$-Mn\mathbf{I}_n\leq\partial_ta^{ij}\leq Mn\mathbf{I}_n.$$
Thus we have
\begin{eqnarray*}
2\mathbf{D}_G+(\frac{\partial_t G-\tilde{\Delta}G}{G}-F)\mathbf{A}
\geq\big(-\frac{2\alpha b\Lambda^2}{t}\langle x\rangle^{\alpha-2}-\frac{4\alpha b n^3\Lambda E}{t}\langle x\rangle^{\alpha-\beta-1}-Mn+\frac{\lambda d}{t}\big)\mathbf{I}_n.
\end{eqnarray*}
Notice that $\alpha-2\leq0$ and $\alpha-\beta-1\leq0$, and if we choose $d=d(n,\Lambda,\lambda,M,E,\alpha)$ large enough, then
$$2\mathbf{D}_G+(\frac{\partial_t G-\tilde{\Delta}G}{G}-F)\mathbf{A}\geq\frac{\lambda d}{2t}\mathbf{I}_n.$$

\textbf{Estimate of $\partial_tF+F(\frac{\partial_tG-\tilde{\Delta}G}{G}-F)$.}\\

By direct calculations we have
\begin{eqnarray*}
\begin{split}
&\partial_tF+F(\frac{\partial_tG-\tilde{\Delta}G}{G}-F)\\
=&\frac{(d+\alpha b\langle x\rangle^{\alpha-2}\partial_ia^{ij}x_j-2)(b\langle x\rangle^{\alpha}-\alpha^2b^2\langle x\rangle^{2\alpha-4}a^{ij}x_ix_j+K)}{t^3}\\
&-\frac{\alpha^2b^2 \langle x\rangle^{2\alpha-4}\partial_ta^{ij}x_ix_j+(d-\alpha b\langle x\rangle^{\alpha-2}a^{ii})(d+\alpha b\langle x\rangle^{\alpha-2}\partial_ia^{ij}x_j-1)}{t^2}\\
&+\frac{\alpha b\langle x\rangle^{\alpha-2}\partial_ta^{ii}}{t}+2\gamma Kt^{-K-2}[K+1-(d+\alpha b\langle x\rangle^{\alpha-2}\partial_ia^{ij}x_j)].\\
\end{split}
\end{eqnarray*}
Notice that
\begin{equation*}
\begin{split}
\langle x\rangle^{\alpha-2}|\partial_ia^{ij}x_j|\leq& C\langle x\rangle^{\alpha-\beta-2}|x|\leq C\langle x\rangle^{\alpha-\beta-1}\leq C,\\
\langle x\rangle^{2\alpha-4}a^{ij}x_ix_j\leq&\Lambda\langle x\rangle^{2\alpha-4}|x|^2\leq\Lambda\langle x\rangle^{2\alpha-2}\leq\Lambda\langle x\rangle^{\alpha},\\
\langle x\rangle^{2\alpha-4}|\partial_ta^{ij}x_ix_j|\leq& C\langle x\rangle^{2\alpha-4}|x|^2\leq C\langle x\rangle^{\alpha},
\end{split}
\end{equation*}
then we have
\begin{eqnarray*}
\begin{split}
&\partial_tF+F(\frac{\partial_tG-\tilde{\Delta}G}{G}-F)\\
\geq& \frac{(d-C)[(b-\alpha^2b^2\Lambda)\langle x\rangle^{\alpha}+K]}{t^3}-\frac{C\langle x\rangle^{\alpha}+(d+C)^2}{t^2}\\
&-\frac{C}{t}+2\gamma Kt^{-K-2}(K-d-C).
\end{split}
\end{eqnarray*}
Recall that $b=\frac{1}{8\Lambda}$, and thus $\alpha^2b^2\Lambda\leq4b^2\Lambda\leq\frac{b}{2}$. If we choose $d$ large enough,
then
\begin{eqnarray*}
\begin{split}
&\partial_tF+F(\frac{\partial_tG-\tilde{\Delta}G}{G}-F)\\
\geq& \frac{(d-C)[\frac{b}{2}\langle x\rangle^{\alpha}+K]-C\langle x\rangle^{\alpha}-(d+C)^2-C}{t^3}+2\gamma Kt^{-K-2}(K-2d)\\
\geq& \frac{(\frac{bd}{2}-C)\langle x\rangle^{\alpha}+(d-C)K-2d^2}{t^3}+2\gamma Kt^{-K-2}(K-2d).
\end{split}
\end{eqnarray*}
We choose $d=\frac{K}{4}$, then
$$\partial_tF+F(\frac{\partial_tG-\tilde{\Delta}G}{G}-F)\geq (\frac{bK}{8}-C)\frac{\langle x\rangle^{\alpha}}{t^3}+\gamma K^2t^{-K-2}\geq \frac{bK\langle x\rangle^{\alpha}}{16 t^3}.$$

\textbf{Estimate of $\tilde{\triangle}F_0$.}\\

Direct calculations show that
\begin{equation}\label{computeF_0}
\begin{split}
\tilde{\triangle}F_0=&\frac{b}{t^2}\tilde{\triangle}(\langle x\rangle^{\alpha})
-\frac{\alpha^2b^2}{t^2}\tilde{\triangle}(\langle x\rangle^{2\alpha-4}a^{ij}_\epsilon x_ix_j)
+\frac{\alpha b}{t}\tilde{\triangle}(\langle x\rangle^{\alpha-2}a^{ii}_\epsilon),
\end{split}
\end{equation}
and
$$
\tilde{\triangle}(\langle x\rangle^{\alpha})=\alpha\langle x\rangle^{\alpha-2}(a^{ii}+\partial_ia^{ij}x_j)
+\alpha(\alpha-2)\langle x\rangle^{\alpha-4}a^{ij}x_ix_j,
$$
\begin{equation*}
\begin{split}
&~\tilde{\triangle}(\langle x\rangle^{2\alpha-4}a^{ij}_\epsilon x_ix_j)=(2\alpha-4)(2\alpha-6)\langle x\rangle^{2\alpha-8}a^{kl}a^{ij}_\epsilon x_ix_jx_kx_l\\
&~~+(2\alpha-4)\langle x\rangle^{2\alpha-6}[(\partial_la^{kl}a^{ij}_\epsilon+2a^{kl}\partial_la^{ij}_\epsilon)x_ix_jx_k
+(4a^{ki}a^{kj}_\epsilon+a^{kk}a^{ij}_\epsilon)x_ix_j]\\
&~~+\langle x\rangle^{2\alpha-4}[(a^{kl}\partial_{kl}a^{ij}_\epsilon+\partial_ka^{kl}\partial_{l}a^{ij}_\epsilon)x_ix_j
+(2\partial_ka^{kj}a^{ij}_\epsilon+4a^{kj}\partial_ka^{ij}_\epsilon)x_i+2a^{ij}a^{ij}_\epsilon],
\end{split}
\end{equation*}
\begin{equation*}
\begin{split}
\tilde{\triangle}(\langle x\rangle^{\alpha-2}a^{ii}_\epsilon)
=&(\alpha-2)(\alpha-4)\langle x\rangle^{\alpha-6}a^{ij}a^{kk}_\epsilon x_ix_j\\
&+(\alpha-2)\langle x\rangle^{\alpha-4}
[(\partial_ja^{ij}a^{kk}_\epsilon+2a^{ij}\partial_ja^{kk}_\epsilon)x_i+a^{ii}a^{kk}_\epsilon]\\
&+\langle x\rangle^{\alpha-2}(a^{ij}\partial_{ij}a^{kk}_\epsilon+\partial_ia^{ij}\partial_ja^{kk}_\epsilon).
\end{split}
\end{equation*}
By Lemma \ref{Lem-mollify} we know that $a^{ij},\nabla a^{ij},a^{ij}_\epsilon,\nabla a^{ij}_\epsilon$ and $\nabla^2a^{ij}_\epsilon$ are all bounded,
then it is easy to verify that
\begin{equation}\label{3parts}
\begin{split}
|\tilde{\triangle}(\langle x\rangle^{\alpha})|\leq& C(\langle x\rangle^{\alpha-1}+\langle x\rangle^{\alpha-2})
\leq C\langle x\rangle^{\alpha-1};\\
|\tilde{\triangle}(\langle x\rangle^{\alpha-4}a^{ij}_\epsilon x_ix_j)|\leq& C(\langle x\rangle^{2\alpha-4}+\langle x\rangle^{2\alpha-3}+\langle x\rangle^{2\alpha-2})\leq C\langle x\rangle^{2\alpha-2};\\
|\tilde{\triangle}(\langle x\rangle^{\alpha-2}a^{ii}_\epsilon)|\leq& C(\langle x\rangle^{\alpha-4}+\langle x\rangle^{\alpha-3}+\langle x\rangle^{\alpha-2})\leq C\langle x\rangle^{\alpha-2}.
\end{split}
\end{equation}
Finally by (\ref{computeF_0}) and (\ref{3parts}) we have
$$
|\tilde{\triangle}F_0|\leq\frac{C}{t^2}(\langle x\rangle^{\alpha-1}+\langle x\rangle^{2\alpha-2}+\langle x\rangle^{\alpha-2})\leq\frac{C\langle x\rangle^{\alpha}}{t^2}.
$$

\textbf{Estimate of $|\nabla(F-F_0)|$.}\\

Since
$$F-F_0=\frac{\alpha^2b^2}{t^2}\langle x\rangle^{2\alpha-4}(a^{ij}_\epsilon-a^{ij})x_ix_j
-\frac{\alpha b}{t}\langle x\rangle^{\alpha-2}(a^{ii}_\epsilon-a^{ii}),$$
then
\begin{equation*}
\begin{split}
\nabla(F-F_0)=&\frac{\alpha^2b^2}{t^2}[(2\alpha-4)\langle x\rangle^{2\alpha-6}(a^{ij}_\epsilon-a^{ij})x_ix_jx
+2\langle x\rangle^{2\alpha-4}(a^{ij}_\epsilon-a^{ij})x_i\nabla x_j\\
&~~~~~~~~~+\langle x\rangle^{2\alpha-4}(\nabla a^{ij}_\epsilon-\nabla a^{ij})x_ix_j]\\
&-\frac{\alpha b}{t}[(\alpha-2)\langle x\rangle^{\alpha-4}(a^{ii}_\epsilon-a^{ii})x
+\langle x\rangle^{\alpha-2}(\nabla a^{ii}_\epsilon-\nabla a^{ii})].
\end{split}
\end{equation*}
Notice that $a^{ij},\nabla a^{ij},a^{ij}_\epsilon$ and $\nabla a^{ij}_\epsilon$ are all bounded, then
\begin{equation}\label{com-nab1}
|\nabla(F-F_0)|\leq\frac{C}{t^2}(\langle x\rangle^{2\alpha-3}
+\langle x\rangle^{2\alpha-4}|\nabla a^{ij}_\epsilon-\nabla a^{ij}||x|^2)+\frac{C}{t}(\langle x\rangle^{\alpha-3}+\langle x\rangle^{\alpha-2}).
\end{equation}
By $2)$ of (\ref{mollify}), when $|x|<1$,
$$|\nabla a^{ij}_\epsilon-\nabla a^{ij}| |x|^2\leq 2M|x|^2\leq 2M,$$
and when $|x|\geq1$,
$$|\nabla a^{ij}_\epsilon-\nabla a^{ij}| |x|^2\leq (2E\langle x\rangle^{-\beta}+E\langle x\rangle^{-\beta})|x|^2
\leq 3E\langle x\rangle^{2-\beta}.
$$
In both cases we have
\begin{equation}\label{com-nab2}
|\nabla a^{ij}_\epsilon-\nabla a^{ij}| |x|^2\leq C\langle x\rangle^{2-\beta}.
\end{equation}
By (\ref{com-nab1}) and (\ref{com-nab2}) we have
$$
|\nabla(F-F_0)|\leq\frac{C}{t^2}(\langle x\rangle^{2\alpha-3}
+\langle x\rangle^{2\alpha-\beta-2})+\frac{C\langle x\rangle^{\alpha-2}}{t}
$$
Since $2\alpha-\beta-2\leq \alpha-1$, then
$$
|\nabla(F-F_0)|\leq\frac{C}{t^2}(\langle x\rangle^{2\alpha-3}
+\langle x\rangle^{\alpha-1})+\frac{C\langle x\rangle^{\alpha-2}}{t}\leq\frac{C\langle x\rangle^{\alpha-1}}{t^2}.
$$

Thus we proved Lemma \ref{estimates1}.
\end{proof}

\bigskip

\noindent {\bf Acknowledgments.}
J. Wu is supported by NSFC under grant 11601373.
L. Zhang is partially supported by NSFC under grant 11471320 and 11631008.

\vspace{1cm}

{\small}

\end{document}